\documentclass[11pt]{article}
\usepackage{amsmath,amsfonts,amssymb,xspace}
\usepackage{verbatim}
\usepackage{mathtools}
\usepackage{enumitem} 
\usepackage{color}
\definecolor{red}{rgb}{.8,0,0}

\usepackage{graphicx}
\usepackage{wrapfig}
\usepackage{tikz}
\usepackage{amsmath}
\usepackage{url}
\usepackage[mathlines]{lineno}
\usepackage{color}

\newtheorem{theorem}{Theorem}[section]
\newtheorem{lemma}[theorem]{Lemma}
\newtheorem{proposition}[theorem]{Proposition}
\newtheorem{corollary}[theorem]{Corollary}

\newtheorem{definition}[theorem]{Definition}

\newenvironment{proof}{\noindent{Proof:}}{\hfill\qed\smallskip}
\newcommand{\qed}{\quad\rule{1.5ex}{1.5ex}}

\newcommand{\ol}{\overline}
\newcommand{\wt}{\widehat}

\newcommand{\mbb}{\mathbb}

\newcommand{\mc}{\mathcal}

\newcommand{\vs}{\vskip 25pt}

\newcommand{\ti}{\textit}

\newcommand{\la}{\langle}
\newcommand{\ra}{\rangle}
\newcommand{\ith}{^\textnormal{th}}
\newcommand{\F}{\mathbb{F}}

\title{Association schemes obtained from the action of the general unitary group on isotropic vectors}
\author{Nathaniel Benjamin\footnote{Nate.Benjamin@dordt.edu, Mathematics and Statistics, Dordt University, Sioux Center, IA 51250, U. S. A.} \ and Sung Y. Song\footnote{sysong@iastate.edu, Department of Mathematics, Iowa State University, Ames, IA 50011, U. S. A.}}

\begin{document} %
\pagenumbering{arabic} \setcounter{page}{1}

\maketitle


\begin{abstract}  An infinite family of association schemes obtained from the general unitary groups acting transitively on the sets of isotropic vectors in the finite unitary spaces are investigated. We compute the parameters and determine the character tables for all nontrivial commutative association schemes belonging to this family. 
\footnote{This work contains a part of the first author's Ph.D. dissertation \cite{Ben}.}
\end{abstract}
\bigskip

\bigskip

MSC-2020 Classification: 05E30 (primary); 20G15, 11E39 (secondary).
\vs


\section{Introduction and preliminaries}

An association scheme of class $d$ is a pair $\mc{X}=\left (X, \{R_i\}_{i\in [d]}\right )$ of a finite set $X$ and a set of non-empty $d+1$ relations $\{R_0, R_1, \dots, R_d\}$  of $X$ such that 
\begin{enumerate}
\item[(1)] {$R_0=\{(x,x): x\in X\}$} is the identity relation;
\item [(2)] {$R_0\cup R_1\cup \cdots \cup R_d=X\times X$} and $R_i\cap R_j= \emptyset$ for $i\neq j$ in $[d]:=\{0, 1, 2, \dots, d\}$; 
\item[(3)] for each $i\in [d]$, {$R_i^{'}=R_{i'}$} for some $i'\in [d]$ where $R_i^{'}:=\{(x,y): (y,x)\in R_i\}$; 
\item[(4)] for each triple $h, i, j\in [d]$, there exists a non-negative integer $p_{ij}^h$ such that for all $(x,y)\in R_h$, the number
$| \{z\in X: (x,z)\in R_i, (z,y)\in R_j\}|$ is equal to $p^h_{ij}$.
\end{enumerate}

The constants $p_{ij}^h$ are called the \ti{intersection numbers} (parameters) of the scheme $\mc{X}$. If an association scheme $\mc{X}=\left (X, \{R_i\}_{i\in [d]}\right )$ satisfies 
that, for all $h, i,j\in [d]$, {$p_{ij}^h=p_{ji}^h$}, 
then it is said to be \ti{commutative}. If for all $i\in [d]$, $R_i'=R_i$, then it is called \ti{symmetric}.

A major source of association schemes is the set of transitive permutation groups (cf. \cite{Ban, BI, BHS, BHSW Zie}). If a finite group $G$ acts on a finite set $\Phi$ transitively, the set of the orbitals (also called 2-orbits) of $G$; that is, the set of the orbits of the action of $G$ on $\Phi\times \Phi$, forms an association scheme. Such an association scheme $\mc{X}=\left (\Phi, \{R_i\}_{i\in [d]}\right )$ where $R_0, R_1, \dots, R_d$ are the orbitals of the permutation group $G$ on $\Phi$ (of rank $d+1$) is called a \emph{Schurian association scheme} (of class $d$) and denoted by $\mc{X}(G, \Phi)$. \

In this paper, we study the Schurian association schemes coming from the transitive action of the finite general unitary groups on the sets of isotropic vectors of the $n$-dimensional unitary space equipped with a non-degenerate Hermitian inner product over the finite field $\mbb{F}_{q^2}$ of order $q^2$ for all $n\ge 2$ and prime powers $q$. Our aim is to provide a complete description of  these Schurian association schemes in terms of their character tables. These association schemes help fill the void in the existing collection of the Schurian association schemes coming from the classical groups acting on various sets of vectors or subspaces of the corresponding geometries. 

For known interesting examples of Schurian association schemes, we refer the readers to Brouwer-Cohen-Neumaier \cite{BCN} for the $P$-polynomial association schemes (distance-regular graphs) corresponding to the dual polar spaces, Brouwer-van Maldeghem \cite{Bro} for known schemes of class 2 (strongly regular graphs),  and Hanaki \cite{Ha} for all Schurian schemes of small orders (orders up to 40 or so). For the list of infinite families of Schurian schemes whose character tables are known, see \cite{ST} and the references there. 

Given a $d$-class commutative association scheme $\mc{X}=(X, \{R_i\}_{i\in [d]})$ of order $n$ (i.e., $|X|=n$),
let $A_i$ denote the $i$th adjacency matrix representing $R_i$; that is, $\{0, 1\}$-matrix whose $(x,y)$-entry is defined by \[\left (A_i\right )_{xy}=\left \{ \begin{array}{ll} 1 & \mbox{if } (x,y)\in R_i\\
0 & \mbox{otherwise.}\\ \end{array}\right . \]
By the definition of commutative association scheme, these matrices satisfy that
\begin{enumerate}
\item[(1)] $A_0=I$, the identity matrix; 
\item[(2)] $A_0+A_1+\cdots +A_d=J$, where $J$ is the all-ones matrix;
\item[(3)] for each $i\in [d]$, $A_i^{'}=A_{i'}$ for some $i'\in [d]$ where $A_i^{'}$ denotes the transpose of $A_i$;
\item[(4)] for any $h, i, j\in [d]$, there exists a constant $p_{ij}^h$ such that
\[A_iA_j=\sum\limits_{h=0}^d p_{ij}^h A_h;\]
\item[(5)] $A_iA_j=A_jA_i$ for every $i, j\in [d]$.
\end{enumerate}

In the full matrix algebra consisting of all $n\times n$ matrices over the field of complex numbers, these adjacency matrices of $\mc{X}$ generate the $(d+1)$-dimensional commutative algebra $\mc{A}:=\left < A_0, A_1, \dots, A_d \right >$ known as the \ti{Bose-Mesner algebra} of $\mc{X}$. 
The algebra $\mc{A}$ being a semi-simple algebra admits central primitive idempotents.
Let $E_0=\frac1n J, E_1, \dots, E_d$ denote the primitive idempotents in $\mc{A}$. Then there exist complex numbers $p_j(i)$ and $q_i(j)$ for $i,j\in [d]$ such that
\[A_j=\sum\limits_{i=0}^d p_j(i)E_i, \quad E_i=\frac{1}{n} \sum\limits_{j=0}^h q_i(j) A_j.\]
The $(d+1)\times (d+1)$ base-change matrices $P$ and $Q$, whose $(i,j)$-entries are $p_j(i)$ and $q_j(i)$, respectively, are called the \ti{1st} and \ti{2nd eigenmatrix} of $\mc{X}$, respectively. The $d+1$ entries $p_j(0), p_j(1), \dots, p_j(d)$, the entries of the column indexed with $j$ of $P$, are the roots of the minimal polynomial of $A_j$ for each $j\in[d]$. The first eigenmatrix $P$ is also called the \emph{character table} of the association scheme. 
We denote the multiplicities of $\mc{X}$ by $m_0, m_1, \dots, m_d$, which are the ranks of  idempotents $E_0, E_1, \dots, E_d$, respectively. Note that $m_i$ is the trace of $E_i$ as the eigenvalues of $E_i$ are 1 and 0. The $i\ith$ \ti{valency}, i.e., the number of elements of $X$ that are in the $i\ith$ associates with $x$ for any fixed $x\in X$, is denoted by $k_i$. Note that $k_i=p_{ii'}^0$. Then we have the following formulae for the parameters and eigenvalues of $\mc{X}$ (cf. \cite[Ch 2]{BI}):
\begin{proposition} \label{formulae}
\[
\begin{array}{llll}
&m_i = n\big (\sum\limits_{j=0}^d \frac{|p_j(i)|^2}{k_j}\big )^{-1} &\qquad &\text{for}\  i\in [d] \\
&p_{ij}^h = \frac1{n\cdot k_h} \sum\limits_{l=0}^d p_i(l) p_j(l) \ol{p_h(l)} m_{l}
&\qquad &\text{for}\ h,i,j\in [d].
\end{array}
\]
where $\ol{a}$ denotes the complex
conjugate of $a$.
\end{proposition}
\begin{proposition} \label{eigen} The character table $P = \big[ p_j(i)\big ]$ of $\mc{X} = (X,\{R_i\}_{i\in [d]})$
satisfies the orthogonality relations:
\[\begin{array}{llll}
&\sum\limits_{j=0}^d \frac1{k_j} p_j(i_1)\cdot \ol{p_j(i_2)} = \delta_{i_1i_2} \frac{n}{m_{i_1}},  &\qquad &i_1,i_2\in [d] \\
&\sum\limits_{i=0}^d m_i p_{j_1}(i) \ol{p_{j_2}(i)} = \delta_{j_1j_2}\cdot n\cdot k_{j_1}, &\qquad &j_1,j_2\in [d]
\end{array}\]
where $\delta_{ij} = \left \{\begin{array}{ll} 1 & \mbox{if } i=j\\ 0 &  \mbox{if } i\neq j\end{array}\right .$.
\end{proposition}

There is another matrix algebra associated with $\mc{X}$. Let $B_i$, $i\in [d]$, be the $i$th \ti{intersection matrix} defined by
$\left (B_i\right )_{jh}=p_{ij}^h.$ Then \[B_iB_j=\sum\limits_{h=0}^d p_{ij}^h B_h.\] It follows that the algebra $\mc{B}:=\left <B_0, B_1, \dots, B_d\right >$ over $\mbb{C}$ is isomorphic to $\mc{A}:=\left < A_0, A_1, \dots, A_d\right >$.  In particular, the minimal polynomial of $B_i$ coincides with that of $A_i$ for each $i\in[d]$, and thus, 
the eigenvalues of $B_i$s are the entries of the character table of the scheme $\mathcal{X}$. 
Therefore, it is possible that the character table can be determined from the eigenvalues of $B_i$ with or without additional information. We will refer to this approach of constructing the character table of $\mc{X}$ as an elementary and `direct approach'. 

We note that there is an `indirect approach' that has been used as well. When $G$ acts transitively on $\Phi$, there is a natural one-to-one correspondence between $\Phi$ and the set $G_x\backslash G$ of the cosets of a point stabilizer $G_x$ for $x\in \Phi$. Thus, the permutation representation of the action $G$ on $\Phi$ is identical to that of the action $G$ on the set $H\backslash G$ of cosets of $H$, where $H=G_x$ for $x\in\Phi$. The permutation character of the representation is the same as the induced character $1^G_{H}$ of $H$ in $G$. We also note that there is a one-to-one correspondence between any two of the following three:
\begin{quote}
(i) the set of orbitals of the permutation group $(G, \Phi)$, \\ (ii) the set of suborbits of $(G,\Phi)$ and\\ (iii) the set of double cosets $Hg_iH$ of $H$ in $G$, $g_i\in G$.    
\end{quote}
It follows that there are links between the following three in terms of associated algebra:
\begin{quote}
(i) the Bose-Mesner algebra of the Schurian scheme $\mc{X}(G, \Phi)$;\\
(ii) the centralizer algebra (Hecke algebra) of $(G, \Phi)$;\\
(iii) the double coset algebra, the subalgebra spanned by the set of simple quantities $\frac{1}{|H|}\sum\limits_{g\in Hg_iH}g$ in group algebra.
\end{quote}

Although we take the direct approach to construct the character table of a Schurian association scheme $\mc{X}(G, \Phi)$ with $G=GU(n,q)$ and $\Phi=\Phi(n, q)$ for suitable $(n, q)$, there is an alternative way to find the entries of the character table for some cases. Namely, its character table may be determined by using the relationship between Bose-Mesner algebra and the associated algebra of the permutation group $(G,\Phi)$ as mentioned above. For instance, 
suppose all orbitals are self-paired (thus, the scheme is symmetric) and the induced character $1^G_H$ is multiplicity-free, so that
\[1^G_H = \chi_0+\chi_1 + \cdots +\chi_d,\]
where $\chi_0=1_G$ and $\chi_1, \chi_2, \dots, \chi_d$ are distinct irreducible characters of $G$. Then the entries $p_j(i)$ of the character table $P=\left [p_j(i) \right ]$ of the scheme $\mc{X}(G, H\backslash G)$ can be explicitly expressed as:
\[p_j(i)\ \ = \ \ \frac{1}{|H|}\sum\limits_{a\in HGH}\chi_i(a) \ \ = \ \ \frac{1}{|H|}\sum\limits_{l} |Hg_jH\cap C_l|\cdot \chi_i(c_l)\]
where $C_0, C_1, \dots, C_l, \dots, C_d$ is the conjugacy classes of $G$ and $c_l$ is the class representative of $C_l$. However, finding the entries $p_j(i)$ of the character table using the above equation is sometimes involved if it is not impossible (cf. \cite{ST} and references in it). 


The highlights of the remaining sections may be described as the following.

\begin{itemize}
\item[Sec 2.] The fact that $GU(n,q)$ is acting transitively on $\Phi(n,q)$ for $n, q$, with $n\ge 2, q\ge 2$ is verified and then,  the orbitals of the action are classified.  
 \item[Sec 3.] 
 The intersection numbers of the association schemes $\mathcal{X}(GU(n,q), \Phi(n, q))$ for all prime powers $q$ and all valid values of $n$ ($n\geq 2$) are calculated. 
 \item[Sec 4.] The fact that association scheme $\mathcal{X}(GU(n,q), \Phi(n, q))$ is commutative if and only if $q=2$ is proved in Theorem \ref{veccom}. 
 \item[Sec 5.] The character table of the commutative scheme $\mathcal{X}(GU(n,2), \Phi(n, 2))$ (for arbitrary $n$) is constructed. 
\end{itemize}

\section{The orbitals of permutation group $GU(n,q)$ on $\Phi(n,q)$} \label{vecorbitsec}

In this section, we shall see that for each prime power $q$ and integer $n\ge 2$, the permutation action of finite unitary group $GU(n,q)$ on the set $\Phi(n,q)$ of isotropic vectors in the finite unitary space $\mbb{F}_{q^2}^n$ is transitive, and thus, its orbitals form an association scheme. Our aim is to give a complete description of orbitals that form the relation set of the association scheme. \footnote{For more information on the facts on finite fields and finite classical geometries that are used in this discussion, we refer the readers to the books by Wan \cite{Wan}, Grove \cite{Gro} and Taylor \cite{Tay}.}

Given a prime power $q$, let $V=\mathbb{F}_{q^2}^n$ be the $n$-dimensional unitary space over $\mathbb{F}_{q^2}$, that is, $V$ is the vector space equipped with a non-degenerate Hermitian inner product on which $GU(n,q)$ is acting.  Recall that the field $\mbb{F}_{q^2}$ has an involution $a\mapsto \bar{a}=a^q$ whose fixed field is $\mbb{F}_q$. We will repeatedly use the fact that: 
(i) For any $\lambda\in \mbb{F}_q^*$, the equation $x\bar{x}=\lambda$ has exactly $q+1$ solutions in $\mathbb{F}_{q^2}^*$.
(ii) For any $\lambda\in \mbb{F}_q$, the equation $x+\bar{x}=\lambda$ has exactly $q$ solutions in $\mathbb{F}_{q^2}$ (cf. \cite{Wan}).

Given row vectors $x=(x_1, x_2, \dots, x_n)$ and  $y=(y_1, y_2, \dots, y_n)$ in $V$, denote the Hermitian inner product of $x$ and $y$ by
\[\la x, y\ra = xy^* = x_1y_1^q+ x_2y_2^q+\cdots + x_n y_n^q.\]
Recall that the set of isotropic vectors in $V$ is $$\Phi(n,q):=\{x\in V-\{0\}: \la x, x\ra = 0\}$$ and the number of isotropic vectors in $V$ is:
\[|\Phi(n,q)|= (q^n-(-1)^n)(q^{n-1}-(-1)^{n-1})\]

In what follows, $GU(n,q)$ and $\Phi(n,q)$ will be often denoted by $G$ and $\Phi$, respectively, for short, once $n$ and $q$ are fixed.
The group action here is described for $x\in \Phi$ and $U\in G$ by $xU$ using typical matrix multiplication. 
We extend the group action of $G$ on $\Phi$ to $G$ acting on $\Phi\times \Phi$ in the obvious way described by $(x,y)U = (xU, yU)$.  

In order to verify that the action of $G$ on $\Phi$ is indeed transitive, we need to recall the following well-known results. Note that a \textit{hyperbolic pair} in $V$ is a pair of vectors $(u,v)$ such that $u$ and $v$ are isotropic, and the Hermitian inner product $\langle u,v\rangle = 1$. The subspace spanned by $\{u,v\}$ is called a \textit{hyperbolic line}.

\begin{proposition}\label{exis}
 Suppose that $L$ is a two-dimensional subspace of $V$ which contains an isotropic vector $u$ with respect to a non-degenerate Hermitian inner product. Then there exists a vector $v$, such that $L = {Span}\{u,v\}$, and $(u,v)$ is a hyperbolic pair.
\end{proposition}

It is known that $V$ can be decomposed into 
$$V = L_1 \bot L_2 \bot \cdots \bot L_m \bot W$$
where each $L_i$ is a hyperbolic line, with $m$ as the Witt index, and such that $W$ does not contain any isotropic vectors. For our case, we can ignore $W$, since we are concerned with isotropic vectors only, though in general for finite fields, dim$(W)$ = 0 or 1 because of the previous proposition, 
which ensures that we can find more isotropic vectors while the dimension of remaining subspace is at least 2. (cf. \cite[pp.116-117]{Tay}).

\begin{proposition}\label{decomp}
For a vector space $V$ that is decomposed as 
$V = L_1 \bot L_2 \bot \cdots \bot L_m \bot W,$
let $(u_i,v_i)$ be hyperbolic pairs that form a basis for each corresponding $L_i$, and let $w$ be a basis for $W$ with $\langle w, w\rangle = 1$. Then $G$ acts regularly on the basis vectors $u_1, \dots u_m, v_1,\dots ,v_m, w$. That is, the action of $G$ on $\Phi$ is transitive and semiregular.
\end{proposition}
From these propositions we can deduce the following.  
\begin{theorem}\label{vectrans}
$GU(n,q)$ acts transitively on $\Phi(n,q)$, for all $n\geq 2$ and all prime powers $q$.
\end{theorem}

\begin{proof}
Let $u,v\in \Phi(n,q)$, and let $(u,u')$ be a hyperbolic pair of $V=\mathbb{F}_{q^2}^n$, which exists by Proposition \ref{exis}. Define $L_1$ as the hyperbolic line spanned by $(u,u')$, and decompose $V$ into hyperbolic lines
$$V = L_1 \bot L_1^\bot = L_1\bot L_2\bot \cdots\bot L_m.$$ Now it is shown that there exists a unitary matrix that takes $u$ to $v$ in case by case.
\begin{enumerate}
\item[(i)] Suppose $v\in L_1^\bot$.
Then again by the above proposition there is a hyperbolic pair $(v,v')$ that spans some $L_i$ with $i\neq 1$. We can choose hyperbolic pairs arbitrarily to form a basis for each of the remaining hyperbolic lines. Then we have a basis for $V$ including $u$ and $v$ as basis elements. By Proposition \ref{decomp}, $GU(n,q)$ acts regularly on these vectors, and thus transitively. Hence, there is a unitary matrix that takes $u$ to $v$.
\item[(ii)] Suppose $v\in L_1$ and $(u,v)$ forms a basis for $L_1$.
Then again we can arbitrarily choose hyperbolic pairs to extend to a basis for $V$ and apply Proposition \ref{decomp} to obtain a unitary matrix that takes $u$ to $v$.\\
\item[(iii)] Suppose $v\in L_1$ and $(u,v)$ does not form a basis for $L_1$.
Then $u$ and $v$ are linearly dependent, i.e. $v = \beta u$. Recall that $(u,u')$ does form a basis for $L_1$, so $(v,u')$ also forms a basis for $L_1$. By (ii), there is a unitary matrix $U\in GU(n,q)$ such that $vU = u'$. Similarly, there is a unitary matrix $U'\in GU(n,q)$ such that $uU'= u'$. Hence, there is a unitary matrix $U'U^{-1}\in GU(n,q)$ that takes $u$ to $v$, since
$$uU'U^{-1} = u'U^{-1} = v.$$
\end{enumerate}
Therefore, $GU(n,q)$ indeed acts transitively on the set of isotropic vectors.
\end{proof}

In order to classify the orbitals of the transitive permutation group $(GU(n,q), \Phi(n, q))$, 
we will first define the desired partition of $\Phi(n,q)\times\Phi(n,q)$, and then proceed to show that these are indeed the orbitals of the action of $GU(n,q)$ on $\Phi(n,q)$; and so the desired association relations.

Let $\alpha$ be an arbitrary but fixed primitive element of $\mathbb{F}_{q^2}$. Consider the following sets where $i,j\in[q^2-2]$ and $\Phi=\Phi(n,q)$.
\[\begin{array}{c}
S_{\alpha^i}(n,q) = \{(x,y)\in \Phi\times \Phi : y = \alpha^i x\}\\
R_{\alpha^j}(n,q)=\{(x,y)\in \Phi\times \Phi : \langle x,y\rangle = \alpha^j\}\\
T(n,q) = \{(x,y)\in \Phi\times \Phi: \langle x,y\rangle = 0, y\notin {Span}\{ x\}\}.\end{array} \qquad \qquad (*)\]
\begin{lemma}
For any $n$ and $q$ and for each $i,j\in[q^2-2]$, the sets $S_{\alpha^i}$, $R_{\alpha^j}$'s, and $T$ defined above partition $\Phi \times \Phi$, i.e.
$$\Phi \times \Phi \ \ = \ \ T \ \cup \ \bigcup_{i=0}^{q^2-2} S_{\alpha^i} \ \cup \ \bigcup_{j=0}^{q^2-2} R_{\alpha^j}.$$
\end{lemma}
\begin{proof}
This follows from the fact that if $\langle x,y\rangle\neq 0$, then $(x,y)$ lies in exactly one of the $R_{\alpha^i}$'s. If $\langle x,y\rangle= 0$ then either $(x,y)\in T$ or $y\in {Span}\{x\}$. In the latter case, $y = \alpha^i x$ for some $i$, implying that $(x,y)\in S_{\alpha^i}$. Hence, each element of $\Phi\times \Phi$ lies in exactly one of the sets above $(*)$.
\end{proof}

Notice that the three types of orbitals: $S_\alpha$'s, $R_\alpha$'s, and $T$, correspond directly with Cases (iii), (ii), and (i) respectively, addressed in the proof of Theorem \ref{vectrans}. The next task is to show that these sets indeed form the orbitals of the group action, that is, showing that $\Phi(n,q)\times \Phi(n,q)$ is not partitioned any finer or courser by the action. This is displayed by showing each set closed under the group action so that none of them need to be combined to form the relations, and that $GU(n,q)$ does not refine any further. 
We will use the following result, known as Witt's Extension Theorem, repeatedly.

\begin{proposition}[Witt's Extension Theorem \cite{Wit}] \label{Witt}  Suppose $V_1, V_2\subseteq \mbb{F}_{q^2}^n$ are subspaces and that there is a linear isomorphism $U':V_1\rightarrow V_2$ such that $\langle xU', yU'\rangle =\langle x, y\rangle $ for any $x,y\in V_1$. Then there exists $U\in GU(n,q)$ such that $U|_{V_1}= U'$.
\end{proposition}
 
\begin{lemma}\label{stran}
    $GU(n,q)$ acts transitively on $S_{\alpha^i}$, for each $i\in [q^2-2]$.
\end{lemma}
\begin{proof}
    Take $(x,y)\in S_{\alpha^i}$, then $y = \alpha^i x$, so for any $U\in GU(n,q)$, we have $yU = \alpha^ixU$. Thus, $(xU,yU) = (x,y)U \in S_{\alpha^i}$. 
    
    To show transitivity, take $(x,y),(w,z)\in S_{\alpha^i}$ we want to show that some $U\in GU(n,q)$ takes $(x,y)$ to $(w,z)$. Since $GU(n,q)$ acts transitively on $\Phi(n,q)$ by Theorem \ref{vectrans}, there exists a $U\in GU(n,q)$ such that $w = xU$. Then $yU = \alpha^ixU = \alpha^iw = z$. Hence, $(x,y)U = (w,z)$, as desired.
\end{proof}

\begin{lemma}\label{rtran}
    $GU(n,q)$ acts transitively on $R_{\alpha^j}$, for each $j\in [q^2-2]$.
\end{lemma}
\begin{proof}
    Take $(x,y)\in R_{\alpha^j}$, then $\langle x,y\rangle = xy^*= \alpha^j$, so for any $U\in GU(n,q)$, 
    $$\langle xU, yU\rangle = xU(yU)^* = xUU^*y^* = xy^* = \alpha^j.$$
    Thus, $(xU,yU) = (x,y)U \in R_{\alpha^j}$.
    
    To show transitivity, take $(x,y),(w,z)\in R_{\alpha^j}$. There exists a linear isomorphism
    $$U':{Span}\{x,y\} \to {Span}\{w,z\}$$
    such that $xU' = w$ and $yU' = z$, as they are both 2-dimensional subspaces. Since 
    $$\langle x,x\rangle = \langle y,y\rangle = \langle w,w\rangle = \langle z,z\rangle = 0 \;\mbox{ and }\; \langle x,y\rangle = \langle w,z\rangle = \alpha^j,$$
    $U'$ is an isometry. Therefore, by Witt's Extension Theorem $U'$ extends to a $U\in GU(n,q)$ mapping $(x,y)\mapsto (w,z)$.
\end{proof}

\begin{lemma}\label{ttran}
    $GU(n,q)$ acts transitively on $T$.
\end{lemma}
\begin{proof}
    Take $(x,y)\in T$, then $\langle x,y\rangle = 0$, but $y\neq \alpha x$ for any $\alpha\in\mbb{F}_{q^2}$. So for any $U\in GU(n,q)$, 
    $$\langle xU, yU\rangle = xU(yU)^* = xUU^*y^* = xy^* = 0.$$
    Now, by way of contradiction, suppose $yU = \alpha xU$ for some $\alpha\in\mbb{F}_{q^2}$, then $yUU^{-1} = \alpha xUU^{-1}$, implying that $y = \alpha x$, a contradiction. Hence, $(xU,yU)\in T$ as desired.
    
    To show transitivity, take $(x,y),(w,z)\in T$, we have that $y\notin {Span}\{x\}$ and $z\notin {Span}\{w\}$, so $Span\{x,y\}$ and $Span\{w,z\}$ are 2-dimensional subspaces. We proceed similarly to before by taking a linear isomorphism $U':{Span}\{x,y\} \to {Span}\{w,z\}$ such that $xU' = w$ and $yU' = z$. Again, this is an isometry since all of the vectors are isotropic, and $\langle x,y\rangle = \langle w,z\rangle = 0.$
    Applying Witt's Extension Theorem gives a map $U\in GU(n,q)$ with $(x,y)\mapsto (w,z)$.
\end{proof}

This shows that the aforementioned sets have the desired structure, but it remains to show that
these are indeed orbitals by showing that they are nonempty. We show that regardless of $n$ and $q$ each $R_\alpha$ and $S_\alpha$ is nonempty; however, the set $T$ is, in fact, empty for $n=2$ and $n=3$, which will be addressed separately.

\begin{lemma}\label{snote}
$S_{\alpha^i}(n,q) \neq \emptyset$ for all $n$ and $q$, and all $i\in [q^2-2]$.
\end{lemma}

\begin{proof}
Consider $x = (1,a,0,\dots ,0)$, where $a\in\mbb{F}_{q^2}$ such that $a\bar a = -1$. It is clear that $x\in \Phi(n,q)$. 


Let $y = \alpha^i x$, then
$$\langle y,y\rangle = \langle\alpha^ix, \alpha^ix\rangle = \alpha^i\alpha^{iq}\langle x,x\rangle = 0.$$
So $y\in \Phi(n,q)$ as well, and $(x,y) \in S_{\alpha^i}$, hence $S_{\alpha^i}\neq \emptyset.$
\end{proof}

\begin{lemma}\label{rnote}
$R_{\alpha^j}(n,q) \neq \emptyset$ for all $n$ and $q$, and all $j\in [q^2-2]$.
\end{lemma}

\begin{proof}
Let $x\in \Phi(n,q)$. Then by Proposition \ref{exis}, we can choose an isotropic vector $y\in \Phi(n,q)$ such that $\langle x,y\rangle =1$. 
Consider $\alpha^jx$. Clearly, $\langle \alpha^jx,\alpha^jx\rangle = \alpha^j(\alpha^j)^q\langle x,x\rangle = 0$,
so $\alpha^jx\in\Phi(n,q)$. Moreover, $\langle \alpha^j x,y\rangle = \alpha^j\langle x,y\rangle = \alpha^j$. Hence, $(\alpha^j x,y)\in R_{\alpha^j}.$
\end{proof}

Now we address the set $T$. 

\begin{lemma}\label{ntwo}
If $n=2$ or $n=3$, then $T(n,q) = \emptyset$ for all $q$. If $n\ge 4$, then $T(n,q) \neq \emptyset.$
\end{lemma}
\begin{proof}
    Notice that an element of $T$ consists of a pair of linearly independent isotropic vectors, meaning they span a two-dimensional totally isotropic subspace. By Theorem 5.7 or \cite{Wan}, we have that $n\ge4$, so $T = \emptyset$ when $n\in\{2,3\}$.
    
    $T$ can be seen to be non-empty for $n\ge 4$ by considering $x=(1,a,0,0,\dots ,0)$ and $y = (0,0,1,a,0,\dots,0)$. Clearly these are isotropic and $\langle x,y\rangle = 0$, but $x\notin{Span}\{y\}$.
\end{proof}

We are finally in a position to identify the association relations of 
$\mc{X}(GU(n,q),\Phi(n,q))$:

\begin{theorem}\label{orb}
For $n\in\{2,3\}$, the relations of $\mc{X}(GU(n,q),\Phi(n,q))$ are the $R_{\alpha^i}(n,q)$'s together with the $S_{\alpha^i}(n,q)$'s for $i\in [q^2-2]$. For $n\geq 4$ the relations are these sets together with $T$. 

Moreover,
\[
    \textnormal{rank}(\mc X(GU(n,q), \Phi(n,q))) = 
    \begin{cases}
    2q^2 - 2 & n\in\{2,3\}\\
    2q^2 - 1 & n\geq 4.
    \end{cases}
\]
\end{theorem}
\begin{proof}
This follows directly from the preceding lemmata.
\end{proof}

\section{Intersection numbers of $\mc{X}(GU(n,q),\Phi(n,q))$}\label{vecintnums}
We now compute the intersection numbers of $\mc X(GU(n,q), \Phi(n,q))$. Theorem \ref{orb} provides enough information  to determine the valencies of $\mc{X}$ for the $n=2$ and $n=3$ cases.

\begin{theorem}\label{val}
For $n\in\{2,3\}$ the valencies $k_i=p^0_{ii'}$ of $\mc X(GU(n,q),\Phi(n,q))$ are
$$\displaystyle \overbrace{1, 1, \dots ,1}^{q^2-1},\;\; \overbrace{\frac{|\Phi(n,q)|}{q^2 - 1} - 1, \  \frac{|\Phi(n,q)|}{q^2 - 1} - 1,\ \dots, \ \frac{|\Phi(n,q)|}{q^2 - 1} - 1}^{q^2 - 1}$$
for $0\leq i\leq D-1$ with $D=2q^2-2$.
\end{theorem}
\begin{proof}
    Note that $|R_\alpha|$ and $|S_\alpha|$ do not depend on $\alpha$. $|S_\alpha| = |\Phi|$ since each vector $x\in\Phi$ uniquely determines $y=\alpha x \in S_\alpha$ for each $\alpha$, so 
    $$k_0 = k_1 = \cdots = k_{q^2-2} = \frac{|S_\alpha|}{|\Phi|} = 1.$$
    The remaining $k_i$'s are equal, say $k$, so $|R_\alpha| = k|\Phi|$. Hence, because $\Phi\times\Phi$ is partitioned by only the $R_\alpha$'s and the $S_\alpha$'s, we can count the number of pairs of isotropic vectors as
    $$|\Phi|(q^2 -1) + k|\Phi|(q^2-1) = |\Phi|^2$$
    so it follows that 
    \begin{eqnarray*}
    k  \ \ = \ \ \frac{|\Phi|}{q^2 - 1} - 1
    \ \ = \ \  \frac{(q^n - (-1)^n)(q^{n-1} - (-1)^{n-1})}{q^2 -1} - 1.
    \end{eqnarray*}
\end{proof}

The presence of $T$ as an additional orbital makes the computation for the valencies of $\mc X$ much more involved for $n\geq 4$.
We start by relabeling the association relations 
with sequential indexing using $R_l$, $l\in [D]$ with $D=2q^2-2$. Namely, by the cyclic nature of the field, we can more concisely articulate the relations as follows.
\[
R_l =
\begin{cases}
    S_{\alpha^l} & \textnormal{if } 0\le l\le q^2-2\\
    R_{\alpha^l} & \textnormal{if } q^2-1 \le l\le 2q^2 -3\\
    T & \textnormal{if } l = D = 2q^2-2.
\end{cases}
\]


Now we want to compute the intersection numbers $p_{ij}^h$ based on the range in which $h,i,$ and $j$ lie. For the sake of simplicity, we will henceforth refer to the ranges in the following way
$$I_1 := [q^2 - 2],\qquad\qquad I_2 := [2(q^2 - 1) - 1]\setminus I_1.$$
The intersection numbers can be computed based only on whether each of $h,i,j$ is in $I_1, I_2$, or equal to $D$. This results in 27 different cases, though many of them are immediate or found similarly to one another. We include one sample calculation of importance below.

\begin{lemma} \label{h1ijD}
    If $h\in I_1$ and $i=j=D$, then $p_{ij}^h = q^2|\Phi(n-2,q)|$.
\end{lemma}
\begin{proof}
Since the intersection numbers do not depend on the choice of $x$ and $y$, we let $x = (1,a,0,\dots,0)$, where again $a$ is an element of the field such that $a\bar{a} = -1$, and let $y = \alpha^h x$.
\begin{eqnarray*}
p_{ij}^h &=& \left\vert\left\{ z \in \Phi(n,q): \langle x,z\rangle = 0, z\notin Span\{x\} ,\langle z,y\rangle = 0, y\notin Span \{z\} \right\}\right\vert\\
&=& \left\vert\left\{ z = \beta x + z': \beta\in \F_{q^2}, z_1' = z_2' = 0,\langle z',z'\rangle = 0, z\notin Span\{x\} , y\notin Span\{z\} \right\}\right\vert\\
&=& \left\vert\left\{ z = \beta x + z': \beta\in\F_{q^2}, z_1' = z_2' = 0, z'\in\Phi(n,q)\right\}\right\vert\\
&=& q^2|\Phi(n-2,q)|.
\end{eqnarray*}
Notice that $z$ satisfying the conditions for $x$ here also ensures that it satisfies the conditions for $y$. There are $q^2$ choices for $\beta$ and $|\Phi(n-2,q)|$ choices for $z'$, since it cannot be the zero vector, otherwise it would be in the span of $x$.
\end{proof}

The following result also assisted in several of these calculations. 
\begin{lemma}\label{primes}\quad
In $\mathcal X (GU(n,q),\Phi(n,q))$, \ for each $l\in [D]$, \ the symmetric conjugate relation $R_{l'}$ of $R_l$ lies in the same range as $R_l$; that is, $l,l'\in I_1$, or $l,l'\in I_2$, or $l = l' = D$.  
\end{lemma}

\begin{theorem}\label{val4}
For $n\geq 4$ the valencies of $\mathcal X(GU(n,q),\Phi(n,q))$ are 
$$ \overbrace{1,\ 1,\ \dots, \ 1}^{q^2 - 1},\;\;\; \overbrace{q^{2n-3}, \ q^{2n-3}, \ \dots, \ q^{2n-3}}^{q^2 - 1},\;\;\; q^2|\Phi(n-2,q)|.$$
\end{theorem}

\begin{proof}
    As before, $k_0 = k_1 = \cdots = k_{q^2-1} = 1.$
    Observe that $k_D = p_{DD}^0$ falls into the case where $i = j = D$ and $h\in I_1$, so by Lemma \ref{h1ijD}, $k_D = q^2|\Phi(n-2,q)|.$
    
    Now the remaining valencies are all the same, say $k$, so $|R_\alpha| = k|\Phi(n,q)|$. Hence, again by double counting the number of pairs of isotropic vectors using the partitioning of $\Phi(n,q)\times\Phi(n,q)$ it follows that
    $$|\Phi(n,q)|(q^2 -1) + k|\Phi(n,q)|(q^2-1) + k_D|\Phi(n,q)| = |\Phi(n,q)|^2.$$
    This gives us 
    \begin{eqnarray*}
    k(q^2 - 1) \ \ = \ \ |\Phi(n,q)| - (q^2 -1) - q^2|\Phi(n-2,q)|
     \ \ = \ \  q^{2n-3}(q^2 - 1).
    \end{eqnarray*}
    Therefore, $k = q^{2n-3}$ as desired.
\end{proof}

Now using the fact that the valencies are as above, we complete the calculation of the intersection numbers that are summarized in the table below.

\renewcommand{\arraystretch}{3.2}
\begin{table} \label{inttable}
\caption{\textbf{Intersection Numbers for $\mc X(GU(n,q),\Phi(n,q))$.}}
\begin{center}
\small
\begin{tabular}{||c|c||c|c|c||}
\hline \hline 
\multicolumn{2}{||c||}{}   & $h\in I_1$ & $h\in I_2$ & $h=D$ \\ \hline \hline
{$i\in I_1$} & $j\in I_1$  &  $ \begin{cases} 1 & \mbox{if} \ h\equiv i+j \\  0 & \mbox{else} \end{cases} $     &      0      &    0   \\ \cline{2-5} 
                            & $j \in I_2$ &      0       &   $ \begin{cases} 1 & \mbox{if } h+i\equiv j  \\  0 & \mbox{else} \end{cases} $  &      0 \\ \cline{2-5} 
                            & $j = D$     &       0      &      0      &    1   \\ \hline \hline
{$i\in I_2$} & $j\in I_1$  &      0       & $ \begin{cases} 1 & \mbox{if } h\equiv i+jq \\  0 & \mbox{else} \end{cases}$  &    0   \\ \cline{2-5} 
                            & $j \in I_2$ & ${\begin{cases} q^{2n-3} & \mbox{if } q(h+i) \equiv j\\ 0 & \mbox{else}  \end{cases}}$  &  $ {\begin{cases} |\Phi(n-2,q)| + 1 & \mbox{if}^*\ i+j \equiv h+t \\
q^{2n-5}+(-q)^{n-3} & \mbox{else} \end{cases}}$   &   $q^{2n-5}$  \\ \cline{2-5} 
                            & $j = D$  &  0 &  $|\Phi(n-2,q)|$  &  $q^{2n - 5}$   \\ \hline \hline
{$i = D$}    & $j\in I_1$  & 0 &  0  &  1  \\ \cline{2-5} 
                            & $j \in I_2$ & 0 & $|\Phi(n-2,q)|$  &   $q^{2n-5}$  \\ \cline{2-5} 
                            & $j = D$  &  $q^2|\Phi(n-2,q)|$ & $|\Phi(n-2,q)|$  & $(q^2 -1)^2 + q^4|\Phi(n-4,q)|$ \\ \hline \hline
\end{tabular}
\end{center}
\vspace{0.5cm}
\renewcommand{\arraystretch}{1}
All the congruences are modulo $q^2-1$, (i.e., should read $\pmod{q^2 -1}$), except for the case when $i,j,h\in I_2$ which is congruent modulo $q+1$, that is, $^*\ i+j\equiv h+t \pmod{q+1}$ \ where 
 \[
 t = \begin{cases}
      0 & \textnormal{ if $q$ is even,}\\ 
      \frac{q+1}{2} & \textnormal{ if $q$ is odd.} 
 \end{cases}
 \]
\end{table}

\renewcommand{\arraystretch}{1}

\newpage
\section{Commutativity of $\mc X(GU(n,q), \Phi(n,q))$}
Now that the intersection numbers have been completely determined for all values of $n$ and $q$, the next step is to construct character tables for those association schemes that are commutative, i.e. $p_{ij}^h = p_{ji}^h$ for all $h,i,j\in [D]$. In fact, this is the case only when $q=2$.

\begin{theorem}\label{veccom}
The association scheme $\mc X(GU(n,q), \Phi(n,q))$ is commutative if and only if $q = 2$.
\end{theorem}
\begin{proof}
For the sufficiency of the statement, take $q=2$ and compute the intersection numbers using the table above in terms of $n$. The size of the intersection matrices does not depend on $n$ when $n\ge 4$, and when $n\in\{2,3\}$ we simply ignore the last row and last column of each matrix, as well as the entire last matrix.\footnote{This does not, however, affect the check for commutativity and the intersection matrices for these two cases of $n$ are shown explicitly in the next section.} The intersection matrices of $\mc X(GU(n,2), \Phi(n,2))$ are as follows, where $s = |\Phi(n-2,2)| = 2^{2n-5}- (-2)^{n-3} - 1$.
\[
B_0 = I, \qquad B_1 =
\left(\begin{array}{rrrrrrr}
0 & 1 & 0 & 0 & 0 & 0 & 0 \\
0 & 0 & 1 & 0 & 0 & 0 & 0 \\
1 & 0 & 0 & 0 & 0 & 0 & 0 \\
0 & 0 & 0 & 0 & 0 & 1 & 0 \\
0 & 0 & 0 & 1 & 0 & 0 & 0 \\
0 & 0 & 0 & 0 & 1 & 0 & 0 \\
0 & 0 & 0 & 0 & 0 & 0 & 1
\end{array}\right),
\qquad B_2 =
\left(\begin{array}{rrrrrrr}
0 & 0 & 1 & 0 & 0 & 0 & 0 \\
1 & 0 & 0 & 0 & 0 & 0 & 0 \\
0 & 1 & 0 & 0 & 0 & 0 & 0 \\
0 & 0 & 0 & 0 & 1 & 0 & 0 \\
0 & 0 & 0 & 0 & 0 & 1 & 0 \\
0 & 0 & 0 & 1 & 0 & 0 & 0 \\
0 & 0 & 0 & 0 & 0 & 0 & 1
\end{array}\right)
\]

\[
B_3 = 
\left(\begin{array}{ccccccc}
0 & 0 & 0 & 1 & 0 & 0 & 0 \\
0 & 0 & 0 & 0 & 0 & 1 & 0 \\
0 & 0 & 0 & 0 & 1 & 0 & 0 \\
2^{2n-3} & 0 & 0 & s+1 & s-(-2)^{n-2}+1 & s-(-2)^{n-2}+1 & 2^{2n-5} \\
0 & 0 & 2^{2n-3} & s-(-2)^{n-2}+1 & s+1 & s-(-2)^{n-2}+1 & 2^{2n-5} \\
0 & 2^{2n-3} & 0 & s-(-2)^{n-2}+1 & s-(-2)^{n-2}+1 & s+1 & 2^{2n-5} \\
0 & 0 & 0 & s & s & s & 2^{2n-5}
\end{array}\right)
\]

\[
B_4 = 
\left(\begin{array}{lllllll}
0 & 0 & 0 & 0 & 1 & 0 & 0 \\
0 & 0 & 0 & 1 & 0 & 0 & 0 \\
0 & 0 & 0 & 0 & 0 & 1 & 0 \\
0 & 0 & 2^{2n-3} & s-(-2)^{n-2}+1 & s+1 & s-(-2)^{n-2}+1 & 2^{2n-5}  \\
0 & 2^{2n-3} & 0 & s-(-2)^{n-2}+1 & s-(-2)^{n-2}+1 & s+1 & 2^{2n-5}  \\
2^{2n-3} & 0 & 0 & s+1 & s-(-2)^{n-2}+1 & s-(-2)^{n-2}+1 & 2^{2n-5} \\
0 & 0 & 0 & s & s & s & 2^{2n-5}
\end{array}\right)
\]

\[
B_5 =
\left(\begin{array}{lllllll}
0 & 0 & 0 & 0 & 0 & 1 & 0 \\
0 & 0 & 0 & 0 & 1 & 0 & 0 \\
0 & 0 & 0 & 1 & 0 & 0 & 0 \\
0 & 2^{2n-3} & 0 & s-(-2)^{n-2}+1 & s-(-2)^{n-2}+1 & s+1 & 2^{2n-5} \\
2^{2n-3} & 0 & 0 & s+1 & s-(-2)^{n-2}+1 & s-(-2)^{n-2}+1 & 2^{2n-5} \\
0 & 0 & 2^{2n-3} & s-(-2)^{n-2}+1 & s+1 & s-(-2)^{n-2}+1 & 2^{2n-5}\\
0 & 0 & 0 & s & s & s & 2^{2n-5}
\end{array}\right)
\]

\[
B_6 =   
\left(\begin{array}{rrrrrrc}
0 & 0 & 0 & 0 & 0 & 0 & 1 \\
0 & 0 & 0 & 0 & 0 & 0 & 1 \\
0 & 0 & 0 & 0 & 0 & 0 & 1 \\
0 & 0 & 0 & s & s & s & 2^{2n-5} \\
0 & 0 & 0 & s & s & s & 2^{2n-5} \\
0 & 0 & 0 & s & s & s & 2^{2n-5} \\
4s & 4s & 4s & s & s & s & 2^{2n-5} - (-2)^{n-1} - 7
\end{array}\right)\;\;\;\;
\]
It can be verified by inspection that $p_{ij}^h = p_{ji}^h$ for all values of $h,i,$ and $j$. 

For the necessity, we show that for each $q\geq 3$, there exist values of $h,i,j$ such that $p_{ij}^h \neq p_{ji}^h$. 
Set $h = q$, $i = q^2 - 1$ and $j = q^2$, which falls into the scenario of $h\in I_1$, and $i,j\in I_2$, so 
\[
p_{ij}^h = 
\begin{cases}
q^{2n-3} & \mbox{if }q(h+i) \equiv j\pmod{q^2 -1}\\
0 & \mbox{else}.
\end{cases}
\]
We have $j = q^2 \equiv 1 \pmod{q^2-1}$, and
$$q(h+i) \equiv q(q + (q^2 -1))\equiv 1 \pmod{q^2-1}$$
so $p_{ij}^h = q^{2n-3}$. On the other hand, for $p_{ji}^h$ we have $i \equiv 0\pmod{q^2-1}$, but
$$q(h+j) \equiv q(q + q^2) \equiv q^2 + q \equiv q+1 \pmod{q^2 - 1},$$
which is only congruent to $0\pmod{q^2 - 1}$ when $q =2$. Hence, $p_{ji}^h = 0$ for all $q\geq 3$. Therefore, $p_{ij}^h \neq p_{ji}^h$, so the association scheme is not commutative.
\end{proof}

From this theorem, character tables for the association scheme can be investigated only when $q=2$, which we do in the subsequent section for generalized $n$.

\section{Character tables of $\mc{X}(GU(n,2),\Phi(n,2))$}\label{vecchars}

We now proceed with the intention of calculating the character table of $\mc{X}(GU(n,q), \Phi(n,q))$ for $q=2$, which requires obtaining the eigenvalues of the intersection matrices. We seek to construct the generalized character table for arbitrary $n$. First, the $n=2$ and $n=3$ cases are dealt with individually, since they produce a character table missing a row and a column of those obtained from larger $n$.

\subsection{Character tables of $\mc X(GU(2,2),\Phi(2,2))$ and $\mc X(GU(3,2),\Phi(3,2))$}\label{n2n3}
First, we compute the intersection matrices for $\mc X = \mc X(GU(2,2),\Phi(2,2))$ and $\mc X = \mc X(GU(3,2),\Phi(3,2))$ using the table of intersection numbers produced in the previous section. We then compute the eigenvalues of each of the intersection matrices for each case of $n$. These are arranged to construct character tables satisfying the orthogonality conditions in Proposition \ref{eigen}. The multiplicities $m_i$ are calculated using the formula 
$$\sum_{h=0}^D \frac1{k_h}p_h(i)\overline{p_h(j)} = \frac{|\Phi(n,2)|}{m_j}\delta_{ij}$$
For each character table, $\omega = \frac12(-1+\sqrt3 \;i)$, i.e. the primitive third root of unity, and the rightmost column entries are the multiplicities, $m_i$.\\

\begin{table}[!h]
\caption{\textbf{Character Table of $\mc X(GU(2,2), \Phi(2,2))$}}
\begin{center}

$P\ = $\ \begin{tabular}{|rrrrrr|r|}
\hline
1 & 1 & 1 & 2 & 2 & 2 & 1 \\
1 & $\omega$ & $\overline{\omega}$ & 2 & $2\overline{\omega}$ & $2{\omega}$ & 1\\
1 & $\overline{\omega}$ & $\omega$ & {2} & $2{\omega}$ & $2\overline{\omega}$ & 1\\
1 & 1 & 1 & $-1$ & $-1$ & $-1$ & 2\\
1 & $\omega$ & $\overline{\omega}$ & $-1$ & $-\overline{\omega}$ & $-{\omega}$ & 2 \\
1 & $\overline{\omega}$ & $\omega$ & $-1$ & $-{\omega}$ & $-\overline{\omega}$ & 2 \\
\hline
\end{tabular}\\
\end{center}
\end{table}

\begin{table}[!h]
\caption{\textbf{Character Table of $\mc X(GU(3,2), \Phi(3,2))$}}
\begin{center}
$P\ = \ $ \begin{tabular}{|rrrrrr|r|}
\hline
1 & 1 & 1 & 8 & 8 & 8 & 1 \\
1 & $\omega$ & $\overline{\omega}$ & $-4$ & $-4\overline{\omega}$ & $-4{\omega}$ & 3\\
1 & $\overline{\omega}$ & $\omega$ & $-4$ & $-4{\omega}$ & $-4\overline{\omega}$ & 3\\
1 & 1 & 1 & $-1$ & $-1$ & $-1$ & 8\\
1 & $\omega$ & $\overline{\omega}$ & $2$ & $2\overline{\omega}$ & $2{\omega}$ & 6 \\
1 & $\overline{\omega}$ & $\omega$ & $2$ & $2{\omega}$ & $2\overline{\omega}$ & 6 \\
\hline
\end{tabular}
\end{center}
\end{table}

Notice that 
\[
|\Phi(n,2)| = \sum_{i=0}^D m_i = 
    \begin{cases}
    9 & \mbox{if } n=2\\
    27  & \mbox{if } n=3,   
    \end{cases}
\]
and the zeroth row gives the valencies, as desired. It is straightforward to verify that the row sums are 0 other than the zeroth row, and that for all $h,i,j\in [D]$,
$$p_{ij}^h = \frac1{|\Phi|k_h} \sum_{l=0}^D p_i(l)p_j(l)\overline{p_h(l)}m_l.$$

We note that the character tables of $\mc X(GU(2,2),\Phi(2,2))$ and $\mc X(GU(3,2),\Phi(3,2))$ are counted in Hanaki's classification of small association schemes with character tables listed in ``as09[10]" and ``as27[403]" in \cite{Ha}.

\subsection{Character table of $\mc X(GU(n,2),\Phi(n,2))$ for $n\geq 4$}

\begin{theorem}\label{vecchartab}
For $n\geq 4$, the character table of $\mc X(GU(n,2),\Phi(n,2))$ is
given by 


\begin{center}
$P \ = \ $\begin{tabular}{|rrrrrrr|r|}
\hline
$1$ & $1$ & $1$ & $2^{2n-3}$ & $2^{2n-3}$ & $2^{2n-3}$ & $2^{2n-3}$ $- (-2)^{n-1} - 4$ & $1$\\
$1$ & $\omega$ & $\overline{\omega}$ & $-(-2)^{n-1}$ & $-(-2)^{n-1}\overline{\omega}$ & $-(-2)^{n-1}{\omega}$ & $0$ & $m_1$\\
$1$ & $\overline{\omega}$ & $\omega$ & $-(-2)^{n-1}$ & $-(-2)^{n-1}{\omega}$ & $-(-2)^{n-1}\overline{\omega}$ & $0$ & $m_2$\\
$1$ & $1$ & $1$ & $-(-2)^{n-2}$ & $-(-2)^{n-2}$ & $-(-2)^{n-2}$ & $3(-2)^{n-2} - 3$ & $m_3$\\
$1$ & $\omega$ & $\overline{\omega}$ & $-(-2)^{n-2}$ & $-(-2)^{n-2}\overline{\omega}$ & $-(-2)^{n-2}{\omega}$ & $0$ & $m_4$\\
$1$ & $\overline{\omega}$ & $\omega$ & $-(-2)^{n-2}$ & $-(-2)^{n-2}{\omega}$ & $-(-2)^{n-2}\overline{\omega}$ & $0$ & $m_5$\\
$1$ & $1$ & $1$ & $-(-2)^{n-3}$ & $-(-2)^{n-3}$ & $-(-2)^{n-3}$ & $3(-2)^{n-3} - 3$ & $m_6$\\
\hline
\end{tabular}
\end{center}
where $\omega$ is the primitive third root of unity.
\end{theorem}
\begin{proof}
First we find the intersection matrices of $\mc X(GU(n,2),\Omega(n,2))$ for arbitrary $n$, as shown in the proof of Theorem \ref{veccom}. We use the help of MATLAB to compute the eigenvalues of each in terms of $n$.  
The eigenvalues are arranged in such a way that they satisfy the orthogonality conditions of Proposition \ref{eigen}.
It is routine to also verify the following formulae for each $h,i,j\in[D]$ using these intersection numbers and the proposed character table as another verification
$$p_{ij}^h = \frac1{|\Phi|\cdot k_h} \sum\limits_{l=0}^D p_i(l) p_j(l) \ol{p_h(l)} m_{l} \qquad\qquad p_i(h)p_j(h) = \sum_{l=0}^D p_{ij}^l p_l(h).$$
\end{proof}

From this we can derive the multiplicities as well.

\begin{corollary}
For $n\geq 4$, the multiplicities of the character table of $\mc X(GU(n,2),\Phi(n,2))$	are
$$m_0 = 1$$
$$m_1 = m_2  = (2^n-(-1)^n)(2^{n-1}-(-1)^{n-1})/9$$
$$m_3 = 4(2^n-(-1)^n)(2^{n-3}-(-1)^{n-3})/9$$
$$m_4 = m_5 = 2(2^n-(-1)^n)(2^{n-1}-(-1)^{n-1})/9$$
$$m_6 =8(2^{n-1}-(-1)^{n-1})(2^{n-2}-(-1)^{n-2})/9$$
\end{corollary}
\begin{proof} 
These multiplicities are all computed again using the formula
$$\sum_{h=0}^D \frac1{k_h}p_h(i)\overline{p_h(j)} = \frac{|\Phi(n, 2)|}{m_j}\delta_{ij}.$$
\end{proof}

This finishes our investigation of the association scheme $\mc{X}(GU(n, q),\Phi(n,q))$ for all values of $n$ and $q$, since all intersection numbers were found, as well as all character tables that exist. 
\vs



\subsection{Fusion schemes of $\mc{X}(GU(n,2), \Phi(n,2))$}
In the classification of association schemes, often knowledge of the character table of its fusion or fission scheme is useful. A fusion scheme of a Schurian association scheme is not necessarily Schurian; however, we observe that $\mc{X}(GU(n,2), \Phi(n,2))$ has two fusion schemes, both of which are Schurian. One is its symmetrization and the other is a 2-class scheme if $n=2$ or $3$, and a 3-class scheme if $n\ge 4$. Note that each of these fusion schemes can be obtained by taking the semidirect product of $GU(n,q)$ and a suitable cyclic group and acting on the set of isotropic vectors, $\Phi(n, q)$.  

Before we discuss the character tables of interest, we first recall some basic facts about the fusion and fission of a commutative schemes  (cf. \cite{Bann, BS1, Hig, JS, FKM, KP, Muz}).

\begin{definition} Let $\mc{X} = \big (X,\{R_i\}_{0\le i\le d}\big )$ and $\wt{\mc{X}} = \big (X,\{\wt{R}_{\alpha}\}_{0\le \alpha \le e}\big )$ be commutative association schemes defined on $X$.  If for every $i\in [d]$, $R_i\subseteq \wt R_\alpha$ for some $\alpha \in [e]$, then we say that $\wt{\mc{X}}$ is a {\it fusion} scheme of $\mc{X}$, and $\mc{X}$ is a {\it fission} scheme of $\wt{\mc{X}}$.  For the notation, we will denote all the symbols belonging to $\wt{\mc{X}}$ by a \ $\widehat{}$ \ placed over the symbols (such as, $\wt{p}_{\alpha\beta}^{\gamma}$,\ $\wt{m}_i$,\ $\wt{k}_i$,\ $\wt{P}$,\ $\wt{p}_j(i)$, etc.) whenever we need to distinguish them from those belonging to $\mc{X}$.
\end{definition}

The following two criteria for fusion will be used repeatedly in this section.
\begin{proposition} For a given scheme $\mc{X} = \big (X,\{R_i\}_{0\le i\le d}\big )$ and a partition $\Lambda = \{\Lambda_\alpha\}_{0\le \alpha \le e}$ of $[d]$ with $\Lambda_0 = \{0\}$,\ $\wt{\mc{X}} = \big (X,\{\wt R_\alpha\}_{0\le \alpha\le e}\big )$ becomes a scheme with the relations defined by $\wt R_\alpha = \bigcup\limits_{i\in \Lambda_\alpha} R_i$, for $\alpha \in [e]$, if and only if
\begin{itemize}
\item[$(i)$]  $\wt R'_\alpha \ \ = \ \ \bigcup\limits_{i\in\Lambda_{\alpha}} R'_i \ \ = \ \ \bigcup\limits_{j\in
\Lambda_{\alpha'}}R_j \ \ = \ \ \wt{R}_{\alpha'}$ \ \ for some $\alpha' \in [e]$, and
\item[$(ii)$]  for any $\alpha,\beta,\gamma\in [e]$, and any $h,k\in \Lambda_\gamma$,
\[\sum\limits_{i\in\Lambda_\alpha}\ \sum\limits_{j\in\Lambda_\beta} p_{ij}^h =\sum\limits_{i\in\Lambda_\alpha}\
\sum\limits_{j\in\Lambda_\beta}\ p_{ij}^k \equiv \widehat{p}_{\alpha\beta}^\gamma.\]
\end{itemize}
\end{proposition}

\begin{proposition}\ (\cite{Bann, JS, Muz}) \label{fusionchar}  Let $\mc{X}
= \big (X,\{R_i\}_{0\le i\le d}\big )$ be a scheme, and $\Lambda = \{\Lambda_\alpha\}_{0\le \alpha\le e}$ be
a partition of $[d]$ such that $\Lambda_0 = \{0\}$.  Suppose for every $\alpha\in[e]$,\
$\bigcup\limits_{i\in\Lambda_\alpha} R_{i'} = \bigcup\limits_{j\in \Lambda_{\alpha'}}R_j$ for some $\alpha'
\in [e]$.  Then $\Lambda$ gives rise to a fusion scheme $\wt{\mc{X}} = \big (X,\{\wt R_\alpha\}_{0\le \alpha \le
e}\big )$ with $\wt R_\alpha = \bigcup\limits_{i\in\Lambda_\alpha} R_i$ if and only if there exists a partition
$\Lambda^* = \{\Lambda_\alpha^*\}_{0\le\alpha\le e}$ of $[d]$ with $\Lambda_0^* = \{0\}$ such
that each $\big (\Lambda_\beta^*,\Lambda_\alpha\big )$-block of the character table $P$ of $\mc{X}$ has a constant row sum. In this case, the constant row sum $\sum\limits_{j\in \Lambda_\alpha} p_j(i)$ for
$i\in \Lambda_\beta^*$ of the block $(\Lambda_\beta^*,\Lambda_\alpha)$ is the $(\beta,\alpha)$-entry
$\widehat{p}_\alpha(\beta)$ of the fusion character table $\wt{P}$.
\end{proposition}

That is to say, character tables of fusion schemes are simply the sum of the characters of the relations that are fused together, removing the duplicate rows that are created.

We note that the fusion scheme obtained by fusing each non-symmetric relation of a scheme $\mc{X}$ with its symmetric conjugate is called the \emph{symmetrization} of $\mc{X}$. It is clear that every non-symmetric commutative association scheme has at least one symmetric fusion scheme, namely, its symmetrization.

\begin{corollary}
    Each of $\mc X(GU(2,2),\Phi(2,2))$ and $\mc X(GU(3,2),\Phi(3,2))$ has its symmetrization. The character tables, denoted by $\bar{P}$, of their symmetrizations are given by 
    \[\bar{P}(GU(2,2),\Phi(2,2)) = \begin{array}{|cccc|r|}\hline
1 & 2 & 2 & 4 & 1 \\
1 & -1 & 2 & -2 & 2\\
1 & 2 &  -1 & -2 & 2\\
1 & -1 & -1  & 1 &  4\\ \hline
\end{array}; \qquad 
\bar{P}(GU(3,2),\Phi(3,2)) = \begin{array}{|cccc|r|}\hline
1 & 2 & 8 & 16 & 1 \\
1 & -1 & -4 & 4 & 6\\
1 & 2 &  -1 & -2 & 8\\
1 & -1 & 2  & -2 &  12\\ \hline
\end{array}.\]
    
    Each of $\mc X(GU(2,2),\Phi(2,2))$ and $\mc X(GU(3,2),\Phi(3,2))$ also has a 2-class fusion scheme whose relations $\wt R_1$ and $\wt R_2$ are obtained as $\wt R_1=R_1\cup R_2$ and $\wt R_2 =R_3\cup R_4\cup R_5$. The character tables for these fusion schemes, all of which are shown to be Schurian association schemes, are given by

\[\wt P = \begin{array}{|ccc|r|}\hline
1 & 2 & 6 & 1 \\
1 & 2 &  -3 & 2\\
1 & -1 & 0 &  6\\ \hline
\end{array};
 \qquad
\wt P =\begin{array}{|ccc|r|}\hline
1 & 2 & 24 & 1 \\
1 & 2 &  -3 & 8\\
1 & -1 & 0 &  18\\ \hline
\end{array}.\]

\end{corollary}
\begin{proof}
    It is clear from Proposition \ref{fusionchar} and the character tables from Section \ref{n2n3}.
\end{proof}

The character tables for these fusion schemes are listed in Hanaki's classification of small association schemes with character tables listed in ``as09[5], as27[383], as09[2]" and ``as27[2]" in \cite{Ha}, respectively.


\begin{corollary}
For $n\ge 4$, the character table of the symmetrization of $\mc{X}(GU(n, 2), \Phi(n,2))$ is given by
{\small
\[\bar P = \begin{array}{|ccccc|r|}\hline
1 & 2 & 2^{2n-3} & 2^{2n-2} & 2^{2n-3}-(-2)^{n-1}-4 & 1 \\
1 & 2 & -(-2)^{n-2} & (-2)^{n-1} & 3(-2)^{n-2}-3 & 4(2^n-(-1)^n)(2^{n-3}-(-1)^{n-3})/9 \\
1 & 2 &  -(-2)^{n-3} & (-2)^{n-2} & 3(-2)^{n-3}-3 & 8(2^{n-1}-(-1)^{n-1})(2^{n-2}-(-1)^{n-2})/9\\
1 & -1 & -(-2)^{n-1}  & (-2)^{n-1} & 0 &  (2^{2n}+(-2)^{n}-2)/9\\
1 & -1 & -(-2)^{n-2} & (-2)^{n-2} & 0 & (2^{2n+1}+2(-2)^{n}-4)/9 \\ \hline
\end{array}.\]
}

The character table of the 3-class fusion scheme whose relations $\wt R_1$, $\wt R_2$ and $\wt R_3$ are defined by $\wt R_1=R_1\cup R_2$, $\wt R_2 =R_3\cup R_4\cup R_5$ and $\wt R_3=R_6$ is given by
\smallskip

\begin{center}
$\wt P \ = \ $\begin{tabular}{|cccc|r|}
\hline
$1$ & $2$ & $3\cdot2^{2n-3}$ & $2^{2n-3}$ $- (-2)^{n-1} - 4$ & $1$\\
$1$ & $2$ & $-3(-2)^{n-2}$ & $3(-2)^{n-2} - 3$ & $4(2^n-(-1)^n)(2^{n-3}-(-1)^{n-3})/9$\\
$1$ & $2$ & $-3(-2)^{n-3}$ & $3(-2)^{n-3} - 3$ & $ 8(2^{n-1}-(-1)^{n-1})(2^{n-2}-(-1)^{n-2})/9$\\
$1$ & $-1$ & $0$ & $0$ & $(2^{2n}+(-2)^n-2)/3$\\
\hline
\end{tabular}.
\end{center}
\end{corollary}

\begin{proof} Omitted.
\end{proof}

\end{document}